\documentclass[12pt]{article}

\usepackage{fullpage}
\usepackage{amsfonts,amsmath,epsf,epsfig,bbm}
\usepackage{amsthm}

\usepackage{graphicx}

\title{A Note on the Subcubes of the n-Cube}

\author{Hans Ulrich Simon}

\newtheorem{theorem}{Theorem}[section]
\newtheorem{definition}[theorem]{Definition}
\newtheorem{lemma}[theorem]{Lemma}
\newtheorem{corollary}[theorem]{Corollary}

\newtheorem{example}[theorem]{Example}

\newcommand{\eset}{\emptyset}

\newcommand{\seq}{\subseteq}
\newcommand{\ra}{\rightarrow}

\newcommand{\nats}{\mathbbm{N}}
\newcommand{\reals}{\mathbbm{R}}

\DeclareMathOperator{\bin}{bin}

\begin{document}

\maketitle

\begin{abstract}
In the year 1990, B\'ela Bollob\'as, Imre Leader and Andrew Radcliffe 
considered the following combinatorial problem: given three 
parameters $k$, $n$ and $q$, find a set of $k$ vertices in the binary
$n$-cube which contains a maximal number of $q$-dimensional subcubes.\footnote{
They also discussed the analogous problem for the grid $[\ell]^n$
at place of the binary $n$-cube.}
It was shown that an optimal solution is given by the $k$ vertices 
which coincide with the binary representations of the number $0,1,\ldots,k-1$.
Two proofs were presented. The proof given in~\cite{BL1990} is particularly
elegant and short. Here we show that also the other proof, the one
given in~\cite{BR1990}, becomes quite simple and short when 
it is combined with a lemma from Graham whose publication dates back 
to 1970~\cite{G1970}. As a second application of Graham's lemma, 
we solve a recursive equation (related to the optimization problem
that we discussed before) that might be considered interesting in 
its own right.
\end{abstract}

\paragraph{Structure of the Paper.}

In Section~\ref{sec:definitions}, we recall some basic definitions and 
fix some notation. The problem, we are mainly concerned with, is
formally described in Section~\ref{sec:main-problem}. 
The set $S^*_{k,n}$ consisting of the $k$ vertices in the $n$-cube
which coincide with the binary representations of $0,1,\ldots,k-1$
is known to be an optimal solution for this problem~\cite{BR1990,BL1990}.
In Section~\ref{sec:graham-based-proof}, we show that Graham's lemma
from~\cite{G1970} can be used to simplify and shorten the proof that 
was given in~\cite{BR1990}. The final Section~\ref{sec:subcubes-recursion}
is devoted to the analysis of a recursive equation which is closely
related to the optimization problem that we considered before.

\section{Definitions and Notations} \label{sec:definitions}

For each pair $a,b$ of non-negative integers, 
we set $[a:b] = \{a,a+1,\ldots,b\}$, where $[a:b] = \eset$ if $a>b$.
Instead of $[1:b]$, we simply write $[b]$. Sets of the form $[a:b]$ 
will be called \emph{intervals}.

As usual, the $n$-cube is the graph with vertex set $B_n = \{0,1\}^n$ 
and an edge between vertices $u$ and $v$ iff $u$ and $v$ differ 
in exactly one bit. We denote by $\bin(i) \in B_n$ with $i \in [0 : 2^n-1]$ 
the binary representation of $i$ (with leading zeros so as to
make it an $n$-bit representation). Throughout the paper,
we denote by $h(i)$ the number of ones in $\bin(i)$ and define 
the function $h_q$ by setting $h_q(i) = \binom{h(i)}{q}$. 

\noindent
For each $r \in [0:n-1]$, the $n$-cube $B_n$ decomposes into
the $(n-1)$-cubes 
\[
B_n(r,0) = \{x \in B_n: x_r=0\}\ \mbox{ and }\ B_n(r,1) = \{x \in B_n: x_r=1\}
\enspace .
\]
More generally, a set $S \seq B_n$ decomposes into
\[
S(r,0) = \{x \in S: x_r=0\}\ \mbox{ and }\ S(r,1) = \{x \in S: x_r=1\}
\enspace .
\]
A \emph{$q$-dimensional subcube of $B_n$} is a subset 
of $B_n$ that is of the form
\[ 
\{x \in B_n: \mbox{$x_i = b(i)$ for all $i \in Q$}\} 
\enspace , 
\]
where $Q \seq [0:n-1]$ is a set of size $n-q$ and $b: Q \ra \{0,1\}$ 
is a function which assigns a bit pattern to the elements of $Q$.
Note that $B_n$ is an $n$-dimensional subcube of itself. $B_n(r,0)$ 
and $B_n(r,1)$ are $(n-1)$-dimensional subcubes of $B_n$. The edges 
(resp.~the vertices) in $B_n$ are the $1$-dimensional 
(resp.~$0$-dimensional) subcubes.

Throughout the paper, we will view the vertices of $B_n$ as
binary representations of the numbers from $0$ to $2^n-1$.
For this reason, we think of the components of $x \in B_n$
as being indexed from $n-1$ downto $0$, 
i.e., $x = (x_{n-1},\ldots,x_1,x_0)$ and $\sum_{i=0}^{n-1}x_i2^i$ 
is the number represented by $x$.

\section{Maximal Number of Subcubes Contained in k-Sets}
\label{sec:main-problem}

Here is the formal definition of the problem we are considering:

\begin{description}
\item[Given] the parameters $n \ge1$, $0 \le q \le n$ 
and $1 \le k \le 2^n$,
\item[find] a set $S \seq B_n$ of size $k$ which contains
a maximal number of $q$-dimensional subcubes.
\end{description}
We refer to this problem as problem $P_q(k,n)$.
The number of $q$-dimensional subcubes contained in $S$ is
denoted by $m_q(S)$. The number of $q$-dimensional subcubes
contained in an optimal solution for problem $P_q(k,n)$ is
denoted $m_q(k,n)$, 
i.e., $m_q(k,n) = \max \{m_q(S): S \seq B_n , |S|=k\}$.

\begin{example}[Trivial Observations]
The problem $P_q(k,n)$ is trivial if $k < 2^q$, if $q=0$, 
or if $k \in \{1,2^n\}$. If $k < 2^q$, then $m_q(k,n) = 0$ 
because no set of a size less than $2^q$ can possibly contain 
a $q$-dimensional subcube. The $0$-dimensional subcubes are
the vertices of the $n$-cube, i.e., the elements of $B_n$. 
Hence any set $S \seq B_n$ of size $k$ contains $k$ $0$-dimensional 
subcubes so that $m_0(k,n) = k$. For obvious reasons, 
setting $k=1$ or $k = 2^n$ also trivializes the problem.
\end{example}

The smallest value of $q$ leading to nontrivial problems $P_q(k,n)$
is the value $1$. $P_1(k,n)$ is the problem of finding 
a set $S \seq B_n$ of size $k$ which contains the largest possible
number of edges (= $1$-dimensional subcubes) of the $n$-cube.
This problem was posed and solved by Sergiu Hart~\cite{H1976}.
He showed that the set $S^*_{k,n} := \{\bin{0},\bin(1),\ldots,\bin(k-1)\}$ 
is an optimal solution of the problem $P_1(k,n)$.\footnote{The main
result in~\cite{H1976} is somewhat more general. Hart defines
so-called ``good $k$-subsets'' of $B_n$ and shows that any good
set is an optimal solution.} Note that the set $S^*_{k,n}$ has 
a dependence on $n$ (albeit a weak one), because the elements
of $S^*_{k,n}$ are elements in $B_n$ ($n$-bit vectors).

As already observed in~\cite{BR1990}, the number of $q$-dimensional
subcubes contained in $S^*_{k,n}$ is easy to determine:
\[
m_q(S^*_{k,n}) = \sum_{j=0}^{k-1}h_q(j) \enspace .
\]
As shown in~\cite{BR1990,BL1990}, no set $S \seq B_n$ of size $k$
can contain more than $m_q(S^*_{k,n})$ many $q$-dimensional subcubes.
In other words: the same set $S^*_{k,n}$ that is the optimal solution
of the problem $P_1(k,n)$ is also an optimal solution of 
the problem $P_q(k,n)$ regardless of how $q$ is chosen. The proof 
given in~\cite{BL1990} is particularly elegant and short. The authors 
define so-called ``$i$-compressions'' ($i=0,1,\ldots,n-1$) which operate 
on sets $S \seq B_n$ of size $k$. It is shown inductively that 
an $i$-compression of $S$ does not contain less $q$-dimensional subcubes 
than $S$ itself. Hence an optimal choice can always be made among 
the sets $S$ which are fully compressed in the sense that each 
$i$-compression maps $S$ to itself. The definition of $i$-compressions 
in~\cite{BL1990} implies that the set $S_{k,n}^*$ is fully compressed.
Moreover, as also shown in~\cite{BL1990}, $S_{k,n}^*$ is the best
choice among all fully compressed sets and therefore an optimal
solution for $P_q(k,n)$.  For details, see~\cite{BL1990}.

The proof given in~\cite{BR1990} is comparably lengthy.
But, as we will show in the next section, this proof can be
considerably simplified when another result, Graham's lemma
from~\cite{G1970}, is brought into play.

\section{A Proof Based on Special Bijections}
\label{sec:graham-based-proof}


\begin{definition}[Special Bijections]
\label{def:special-bijection} 
Let $I = [i_0 : i_0+s-1]$ and $J = [j_0 : j_0+s-1]$ with $j_0 > i_0$ 
be two sub-intervals of $[0 : 2^n-1]$, each of size $s$.
A bijection $P:I \ra J$ is said to be \emph{special} if $h(i) \le h(P(i))$
for each $i \in I$ and, in addition,  all these inequalities are strict in 
case that $I$ and $J$ are non-overlapping (i.e., in case that $i_0+s-1 < j_0$).
\end{definition}

\noindent
We will make use of the following result (dating back to the year 1970):

\begin{lemma}[\cite{G1970,JT1999}] \label{lem:graham}
Let $I = [i_0 : i_0+s-1]$ and $J = [j_0 : j_0+s-1]$ with $j_0 > i_0$
be two sub-intervals of $[0 : 2^n-1]$, each of size $s$. If $i_0=0$,
then there exists a special bijection $P:I \ra J$.
\end{lemma}

\noindent
Here are some easy-to-prove facts regarding special bijections: 

\begin{lemma} \label{lem:special-bijection}
Let $I = [i_0 : i_0+s-1]$ and $J = [j_0 : j_0+s-1]$ with $j_0 > i_0$ 
be sub-intervals of $[0 : 2^n-1]$ which allow for a special bijection. 
Let $g:\nats_0 \ra \reals$ be a non-decreasing function. Then the following 
hold:
\begin{enumerate}
\item
$\sum_{i \in I}g(h(i)) \le \sum_{j \in J}g(h(j))$.
\item
If $I$ and $J$ are non-overlapping and $g$ is strictly increasing, 
then $\sum_{i \in I}g(h(i)) < \sum_{j \in J}g(h(j))$.
\item
For each $0 \le q \le n$, let $h_q:[0:2^n-1] \ra \nats_0$ be the function
given by $h_q(i) = \binom{h(i)}{q}$. Suppose that $I$ and $J$ are
non-overlapping. Then, for all $1 \le q \le n$, 
we have that $\sum_{i \in I}(h_q(i)+h_{q-1}(i)) \le \sum_{j \in J}h_q(j)$.
\end{enumerate}
\end{lemma}

\begin{proof}
Let $P:I \ra J$ be a special bijection, i.e., $h(i) \le h(P(i))$
for each $i \in I$ and all of these inequalities are strict in case 
that $I$ and $J$ are non-overlapping. Since $g$ is non-decreasing, 
it follows that
\[
\sum_{i \in I}g(h(i)) \le \sum_{i \in I}g(h(P(i))) = \sum_{j \in J}g(h(j))
\]
and the inequality between the first two terms will be strict 
if $I$ and $J$ are non-overlapping and $g$ is strictly increasing. 
This proves the first two statements of the lemma. The following 
calculation shows the correctness of the third statement:
\begin{eqnarray*}
\sum_{i \in I}(h_q(i) + h_{q-1}(i)) & = & 
\sum_{i \in I}\left(\binom{h(i)}{q} + \binom{h(i)}{q-1}\right)
= \sum_{i \in I}\binom{h(i)+1}{q} \\
& \le & \sum_{i \in I}\binom{h(P(i))}{q} 
= \sum_{j \in J}\binom{h(j)}{q} = \sum_{j \in J}h_q(j)
\enspace .  
\end{eqnarray*}
The inequality in this calculation holds for two reasons.
First, the third statement of the lemma makes the assumption
that $I$ and $J$ are non-overlapping. Hence $h(i)+1 \le h(P(i))$
for each $i \in I$. Second, the function $m \mapsto \binom{m}{q}$
is non-decreasing. These observations complete the proof.
\end{proof}

The following theorem establishes $S^*_{k,n}$ as an optimal
solution for the problem $P_q(k,n)$. The proof that we present
below essentially coincides with the easy part of the proof 
from~\cite{BR1990}. The hard part of that proof will be simply
replaced by a cross-reference to Lemmas~\ref{lem:graham} 
and~\ref{lem:special-bijection}.

\begin{theorem}[\cite{BR1990,BL1990}] \label{th:main-result}
For any choice of $n \ge 1$, $1 \le k \le 2^n$ and $0 \le q \le n$, 
let $P_q(k,n)$ be the problem of finding a set $S \seq B_n$ of size $k$
that contains the largest possible  number of $q$-dimensional subcubes. 
For this problem, the set $S^*_{k,n} = \{\bin(0),\bin(1),\ldots,\bin(k-1)\}$ 
with $m_q(S^*_{k,n}) = \sum_{j=0}^{k-1}h_q(j)$ is an optimal solution. 
\end{theorem}

\begin{proof} 
The proof proceeds by induction on $k$. As the result is trivial
for $k=1$, one may assume that $k\ge2$. Let $S \seq B_n$ be any set 
consisting of $k$ vertices in $B_n$. Choose an index $r \in [0:n-1]$ 
such that the set $S$ has non-empty intersections with both of the 
subcubes $B_n(r,0)$ and $B_n(r,1)$ so that the sets $S(r,0)$ 
and $S(r,1)$ are both non-empty. Set $k_0 := |S(r,0)|$ 
and $k_1 = |S(r,1)|$.  For reasons of symmetry, one may assume 
that $k_1 \le k_0$. With these notations, the following hold:
\begin{eqnarray*}
m_q(S) & \le & 
m_q(S(r,0)) + m_q(S(r,1)) + m_{q-1}(S(r,1)) \\
& \le & 
\sum_{j=0}^{k_0-1}h_q(j) + \sum_{i=0}^{k_1-1}(h_q(i) + h_{q-1}(i)) \\
& = &
\sum_{j=0}^{k-1}h_q(j) - \sum_{j=k_0}^{k-1}h_q(j) + 
\sum_{i=0}^{k_1-1}(h_q(i) + h_{q-1}(i)) \\
& \stackrel{(*)}{\le} & \sum_{j=0}^{k-1}h_q(j) 
= m_q(S^*_{k,n}) \enspace .
\end{eqnarray*}
The first inequality holds because each $q$-dimensional subcube
that is contained in $S$ is either contained in $S(r,0)$, or contained 
in $S(r,1)$, or composed of a $(q-1)$-dimensional subcube in $S(r,1)$ 
and a corresponding $(q-1)$-dimensional subcube in $S(r,0)$. The second 
inequality results from an application of the inductive hypothesis. 
The final inequality (marked $(*)$) is the non-trivial one. It is 
proven in~\cite{BR1990} by a series of technical results. But note 
that this inequality is immediate from Lemma~\ref{lem:graham} and the 
third assertion in Lemma~\ref{lem:special-bijection}. The latter 
assertion applies because the intervals $I = [0:k_1-1]$ and $J = [k_0:k-1]$ 
are non-overlapping. Hence the proof of Theorem~\ref{th:main-result}
is complete.
\end{proof}

\noindent
For $r \in [n]$ and $b \in \{0,1\}$ such that $1 \le |S(r,b)| \le |S(r,1-b)|$,
we refer to
\begin{equation} \label{eq:decomposition}
m_q(S) \le m_q(S(r,1-b)) + m_q(S(r,b)) + m_{q-1}(S(r,b))
\end{equation}
as the \emph{three-term decomposition of $m_q(S)$ along dimension $r$}. 
The decomposition is said to be \emph{exact} if~(\ref{eq:decomposition}) 
holds with equality.

\section{A Recursion Involving a Max-Expression} 
\label{sec:subcubes-recursion}

Here comes one more application of Lemmas~\ref{lem:graham}
and~\ref{lem:special-bijection}:

\begin{lemma} \label{lem:recursion}
For $1 \le k \le 2^n$ and $1 \le q \le n$, consider the following 
inductive definition of a function $F_q(k)$:
\begin{eqnarray*}
F_q(1) & = & 0 \enspace . \\
F_q(k) & = & \max_{1 \le k' \le k/2}\left(F_q(k') + F_q(k-k') + F_{q-1}(k')\right)
\enspace .
\end{eqnarray*}
Then 
\begin{equation} \label{eq:recursion}
F_q(k) = \sum_{i=0}^{k-1}h_q(i) \enspace . 
\end{equation}
\end{lemma}

\begin{proof}
We first observe that an inequality of the form $i < 2^\ell$ implies
that $h(2^\ell + i) = 1+h(i)$ and therefore
\begin{equation} \label{eq:pascal}
h_q(2^\ell+i) = \binom{h(2^\ell+i)}{q} = \binom{1+h(i)}{q} =
\binom{h(i)}{q} + \binom{h(i)}{q-1} = h_q(i) + h_{q-1}(i) \enspace.
\end{equation}
The proof of equation~(\ref{eq:recursion}) will proceed by induction 
over $k$. Clearly $F_q(1) = 0 = \binom{0}{q} = h_q(0)$. 
Hence~(\ref{eq:recursion}) holds for $k=1$.
For $k\ge2$, choose $r$ such that $2^r < k \le 2^{r+1}$.
The inequality $F_q(k) \ge \sum_{i=0}^{k-1}h_q(i)$
can be obtained as follows:
\begin{eqnarray*}
F_q(k) & \ge & F_q(k-2^r) + F_q(2^r) + F_{q-1}(k-2^r) \\
& = &  
\sum_{i=0}^{k-2^r-1}(h_q(i) + h_{q-1}(i)) + \sum_{i=0}^{2^r-1}h_q(i) \\
& = &
\sum_{i=2^r}^{k -1}h_q(i) + \sum_{i=0}^{2^r-1}h_q(i) = 
\sum_{i=0}^{k-1}h_q(i) \enspace .
\end{eqnarray*}
The first equation in this calculation is true by the inductive hypothesis.
Furthermore, since $k \le 2^{r+1}$, we have that $k-2^r-1 < 2^r$.
Hence equation~(\ref{eq:pascal}) applies and we obtain the second equation
in the above calculation. \\
Choose now $1 \le k' \le k/2$ 
such that $F_q(k) = F_q(k') + F_q(k-k') + F_{q-1}(k')$.
We still have to show that $F_q(k) \le \sum_{i=0}^{k-1}h_q(i)$.
By another application of the inductive hypothesis, we see that
it is sufficient to show that
\begin{equation} \label{eq:subcube-recursion}
\sum_{i=0}^{k'-1}(h_q(i)+h_{q-1}(i)) \le 
\sum_{j=0}^{k-1}h_q(j) - \sum_{j=0}^{k-k'-1}h_q(j) = 
\sum_{j=k-k'}^{k-1}h_q(j) \enspace .
\end{equation}
Thanks to $k' \le k/2$, the intervals $I := [0: k'-1]$
and $J := [k-k' : k-1]$ are non-overlapping. It follows from
Lemma~\ref{lem:special-bijection} that $(I,J)$ admits a 
special bijection. An inspection of the third statement in 
Lemma~\ref{lem:graham} reveals that~(\ref{eq:subcube-recursion}) 
is valid.
\end{proof}

Comparing the recursive definition of $F_q(k)$ in
Lemma~\ref{lem:recursion} with the three-term decomposition of $m_q(S)$,
it becomes evident that $F_q(k) \ge m_q(S)$ holds for all sets $S \seq B_n$
of size $k$. Since $F_q(k) = \sum_{j=0}^{k-1}h_q(j) = m_q(S^*_{k,n})$,
Lemma~\ref{lem:recursion} gives us another option for proving 
the optimality of $S^*_{k,n}$. But here we are more interested in
using optimal solutions for problem $P_q(k,n)$ in order to find 
maximizers $1 \le  k' \le k/2$ of $F_q(k') + F_q(k-k') + F_{q-1}(k')$.
An inspection of Lemma~\ref{lem:recursion} reveals that 
a set $S \seq B_n$ of size $k$ satisfies the optimality
condition $m_q(S) = F_q(k)$ only if the following hold for each $r \in [n]$ 
such that $S(r,0) , S(r,1) \neq \eset$:
\begin{enumerate}
\item
The three-term decomposition of $m_q(S)$ along dimension $r$ is exact.
\item
The number $k' = \min\{|S(r,0)| , |S(r,1)|\}$ is a maximizer
of $F_q(k') + F_q(k-k') + F_{q-1}(k')$ subject to $1 \le k' \le k/2$.
\end{enumerate}

As in~\cite{A2013}, we say that $(k-k_1,k_1)$ with $1 \le k_1 \le k/2$ 
is a \emph{hypercubic partition of $k$} if there exists a bit position $r$
such that $k_1 = |S^*_{k,n}(r,1)|$. In other words:
$k_1$ of the vertices $\bin(0),\bin(1),\ldots,\bin(k-1)$ 
belong to the subcube $B_n(r,1)$ and the remaining $k-k_1$ of them
belong to the subcube $B_n(r,0)$. Since $S^*_{k,n}$ 
contains $F_q(k)$  many $q$-dimensional subcubes,
we obtain the following result:
 
\begin{corollary}
Let $F_q$ be the function from Lemma~\ref{lem:recursion}. 
Let $(k-k_1,k_1)$ with $1 \le k_1 \le k/2$ be a hypercubic partition of $k$. 
Then $k' = k_1$ is a maximizer of $F_q(k') + F_q(k-k') + F_{q-1}(k')$
subject to $1 \le k' \le k/2$,
i.e., $F_q(k) = F_q(k-k_1) + F_q(k_1) + F_{q-1}(k_1)$. 
\end{corollary}

For instance, $(\lceil k/2 \rceil , \lfloor k/2 \rfloor)$ is
a hypercubic partition of $k$ as witnesses by $j=0$ (the position 
of the least significant bit). If $r$ satisfies $2^r < k \le 2^{r+1}$,
then $(2^r , k-2^r)$ is a hypercubic partition too. This is witnessed
by choosing $j=r$ (the position of the most significant bit). 

For the special case $q=1$, it is shown in~\cite{A2013} that $k' = k_1$
is a maximizer of $F_1(k') + F_1(k-k') + k'$ if and only if $(k-k_1,k_1)$
is a hypercubic partition of $k$. The ``only-if'' direction is not
generally true for $q > 1$.


\end{document}